\documentclass[11pt,letterpaper,reqno]{amsart}

\usepackage{amsmath,amssymb,amsthm}
\usepackage{enumerate}
\usepackage{multirow}
\usepackage{tikz}
\usepackage{amsfonts}
\usepackage{array}
\usepackage{slashbox}
\usepackage{float}

\newtheorem{theorem}{Theorem}[section]
\newtheorem{proposition}[theorem]{Proposition}

\theoremstyle{definition}
\newtheorem{remark}[theorem]{Remark}

\def\cqfd{
{\hfill
\kern 6pt\penalty 500
\raise -1pt\hbox{\vrule\vbox to 5pt{\hrule width 4pt
\vfill\hrule}\vrule}}
\break}

\title{Optimal and maximal singular curves}

\author[Aubry]{Yves Aubry}
\address[Aubry]{Institut de Math\'ematiques de Toulon, Universit\'e de Toulon, France}
\address[Aubry]{Institut de Math\'ematiques de Marseille, CNRS-UMR 7373, Aix-Marseille Universit\'e, France}
\email{yves.aubry@univ-tln.fr}

\author[Iezzi]{Annamaria Iezzi}
\address[Iezzi]{Institut de Math\'ematiques de Marseille, CNRS-UMR 7373, Aix-Marseille Universit\'e, France}
\email{annamaria.iezzi@univ-amu.fr}

\begin{document}

\maketitle

\begin{abstract}
Using an Euclidean approach, we prove a new upper bound for the number of closed points of degree 2 on a smooth absolutely irreducible projective algebraic curve defined over the finite field $\mathbb F_q$.
This bound enables us to provide explicit conditions on $q, g$ and $\pi$  for the non-existence of absolutely irreducible projective algebraic curves defined over $\mathbb F_q$ of geometric genus $g$, arithmetic genus $\pi$ and with $N_q(g)+\pi-g$ rational points.
Moreover, for $q$ a square, we study the set of pairs $(g,\pi)$ for which there exists a maximal absolutely irreducible projective algebraic  curve defined over $\mathbb F_q$ of geometric genus $g$ and arithmetic genus $\pi$, i.e. with $q+1+2g\sqrt{q}+\pi-g$ rational  points.

\bigskip

MSC[2010] 14H20, 11G20, 14G15.
\end{abstract}


\section{Introduction}
Throughout the paper, the word \emph{curve} will stand for an absolutely irreducible projective algebraic curve and $\mathbb F_q$ will denote the finite field with $q$ elements.


Let $X$ be a curve defined over $\mathbb F_q$  of geometric genus $g$ and arithmetic genus $\pi$. The first author  and Perret showed in \cite{A-P} that
the number $\sharp X(\mathbb F_q)$ of rational points over ${\mathbb F}_q$ on $X$ verifies:
\begin{equation}\label{APbound}\sharp X(\mathbb F_q)\leq q+1+g[2\sqrt{q}]+ \pi-g.\end{equation}

Furthermore if we denote  by $N_q(g,\pi)$ the maximum number of rational points on a curve defined over $\mathbb F_q$ of geometric genus $g$ and arithmetic genus $\pi$, it is  proved in \cite{A-I} that:
$$
N_q(g)\leq N_q(g,\pi)\leq N_q(g) +\pi-g,
$$
where $N_q(g)$ classically denotes the maximum number of rational points over $\mathbb F_q$ on a smooth curve defined  over $\mathbb F_q$ of genus $g$.

The curve $X$  is said to be  \emph{maximal} if it attains the bound ($\ref{APbound}$).
This definition for non-necessarily smooth curves has been introduced in \cite{A-I} and recovers the classical definition of maximal curve when $X$ is smooth.

More generally (see \cite{A-I}), $X$ is said to be  $\delta$\emph{-optimal} if
$$\sharp{X}({\mathbb F}_q)= N_q(g)+ \pi-g.$$

Obviously the set of maximal curves is contained in that  of $\delta$-optimal ones.

\bigskip

In  \cite{A-I} we were interested in the existence of   $\delta$-optimal and maximal curves of prescribed geometric and arithmetic genera.
 Precisely, we proved (see Theorem 5.3 in  \cite{A-I}):
 \begin{equation}\label{iff}N_q(g,\pi)= N_q(g)+\pi-g
\ \ \Longleftrightarrow\ \  g\leq \pi \leq g+B_2({\mathcal X}_q(g)),\end{equation}
where ${\mathcal X}_q(g)$ denotes the set of optimal smooth curves defined over $\mathbb F_q$ of genus $g$ (i.e. with $N_q(g)$ rational points) and $B_2({\mathcal X}_q(g))$  the maximum number of closed points of degree 2 on a curve of ${\mathcal X}_q(g)$.

The quantity $B_2({\mathcal X}_q(g))$ is easy to calculate for $g$ equal to $0$ and $1$ and also for those $g$ for which $N_q(g)=q+1+g[2\sqrt{q}]$ (see Corollary 5.4, Corollary 5.5 and Proposition 5.8 in \cite{A-I}), but is not explicit in the general case.

\bigskip

The first  aim of this paper is  to provide upper and lower bounds for $B_2({\mathcal X}_q(g))$.
 For this purpose we will follow
the Euclidean approach developed by Hallouin and Perret in \cite{H-P} and recalled in Section \ref{recall}. These new bounds will allow us
to provide explicit conditions on $q, g$ and $\pi$  for the non-existence of $\delta$-optimal curves and to
determine some exact values of $N_q(g,\pi)$ for specific  triplets $(q,g,\pi)$.


Secondly, in Section \ref{spectrum}, we will assume  $q$ to be square and, as it is done in the smooth case, we will study the genera spectrum of  maximal curves defined over $\mathbb F_q$, i.e. the set of couples $(g,\pi)$, with $g,\pi \in \mathbb N $ and $g\leq \pi$, for which  there exists a maximal curve defined over $\mathbb F_q $  of geometric genus  $g$ and arithmetic genus $\pi$.



\section{The Hallouin-Perret's approach}\label{recall}

Let $X$ be a smooth curve defined over $\mathbb F_q$ of genus $g>0$.


For every positive $ n\in \mathbb N$,  we associate to $X$ a $n$-tuple $(x_1,\ldots,x_n)$ defined as follows:
\begin{equation}\label{xi}x_i:=\frac{(q^i+1)-\sharp X(\mathbb F_{q^i})}{2g\sqrt{q^i}}, \quad i=1,\ldots, n.\end{equation}
The Riemann Hypothesis proved by Weil gives that
\begin{equation}\label{weil}\sharp X(\mathbb F_{q^i})=q^i +1-\sum_{j=1}^{2g}\omega_j^i,\end{equation}
where $\omega_1,\ldots, \omega_{2g}$ are  complex numbers of absolute value $\sqrt q$. Hence one easily gets $|x_i|\leq 1$ for all $i=1,\ldots, n$, which means that the $n$-tuple $(x_1,\ldots,x_n)$ belongs to the hypercube
\begin{equation}\label{cube}\mathcal C_n=\{(x_1,\ldots,x_n)\in \mathbb R^n|-1\leq x_i\leq1,\,\ \forall\,i=1,\ldots,n\}.\end{equation}

The Hodge Index Theorem implies that the intersection pairing on the Neron-Severi space over $\mathbb R$ of the smooth algebraic surface $X\times X$ is anti-Euclidean on the orthogonal complement of the trivial plane generated by the horizontal and vertical classes. Hallouin and Perret used  this fact in   \cite{H-P}
to obtain that the following matrix
$$G_n=\left(\begin{matrix}
1 &x_1 & \cdots& x_{n-1} & x_n\\
x_1 &1&x_1&\ddots& x_{n-1} \\
\vdots &\ddots& \ddots &\ddots &\vdots \\
x_{n-1} &\ddots& \ddots &1 &x_1 \\
x_n &x_{n-1}& \cdots &x_1 &1 \\
\end{matrix}\right)$$
is
a Gram matrix
and then   positive semidefinite (the $x_i$'s are interpreted as inner products of
 normalized Neron-Severi classes of the iterated Frobenius morphisms).




Now, a matrix is positive semidefinite if and only if all the principal minors are non-negative. This fact implies that the $n$-tuple $(x_1,\ldots,x_n)$ has to belong to the set
\begin{equation}\label{omega}\mathcal W_n=\{(x_1,\ldots,x_n)\in \mathbb R^n|\,G_{n,I}\geq 0,\,\forall\,I\subset\{1,\ldots,n+1\}\},\end{equation}
where  $G_{n,I}$ represents the principal minor of $G_n$ obtained by deleting the lines and columns whose indexes are not in $I$.

\vspace{5 pt}
To these relations that come from the geometrical point of view, one can  add the arithmetical contraints resulting from the obvious inequalities  $\sharp X(\mathbb F_{q^i})\geq \sharp X(\mathbb F_q)$, for all $i\geq 2$. It follows that, for all $i\geq 2$, $$x_i\leq \frac{x_1}{q^{\frac{i-1}{2}}}+\frac{q^{i-1}-1}{2gq^{\frac{i-2}{2}}}.$$
Setting  $$h_i^{q,g}(x_1,x_i)=x_i-\frac{x_1}{\sqrt{q}^{i-1}}-\frac{\sqrt{q}}{2g}\left(\sqrt{q}^{i-1}-\frac{1}{\sqrt{q}^{i-1}}\right)$$
one gets that the $n$-tuple $(x_1,\ldots,x_n)$ has to belong to the set
\begin{equation}\label{acca}\mathcal H_n^{q,g}=\{(x_1,\ldots,x_n)\in \mathbb R^n|h_i^{q,g}(x_1,x_i)\leq0, \textrm{ for all }2\leq i\leq n\}.\end{equation}
We  assume that $\mathcal H_1^{q,g}=\mathbb R$.
\begin{remark}\label{remark1}
We have  $h_i^{q,g}(x_1,x_i)=0$ if and only if  $\sharp X(\mathbb F_q)=\sharp X(\mathbb F_{q^i}$).
\end{remark}

Finally we obtain (Proposition 16 in \cite{H-P}) that if $X$ is a smooth curve defined over $\mathbb F_q$ of genus $g>0$, then its \mbox{associated} $n$-tuple $(x_1,\ldots,x_n)$ belongs to $ \mathcal C_n \cap \mathcal W_n \cap \mathcal H_n^{q,g}$, where $\mathcal C_n, \mathcal W_n, \mathcal H_n^{q,g}$ are respectively defined by $(\ref{cube})$, $(\ref{omega})$ and $(\ref{acca})$.

\bigskip

Fixing $n=1,2,3,\ldots$ we find bounded subsets of $\mathbb R^n$ to which the $n$-tuple $(x_1,\ldots,x_n)$ belongs. Hence we can  obtain lower or upper bounds for  $\sharp X(\mathbb F_{q^i} )$ by remarking that any  lower bound for $x_i$ corresponds to an upper bound for $\sharp X(\mathbb F_{q^i} )$ and, conversely, any upper bound for $x_i$ corresponds to a lower bound for $\sharp X(\mathbb F_{q^i} )$.


Hallouin and Perret  showed in  \cite{H-P} that, increasing the dimension $n$, the set $\mathcal C_n\cap \mathcal W_n\cap \mathcal H_n^{q,g}$ provides a better and better  lower bound for $x_1$  (and hence a better and better upper bound for $\sharp X(\mathbb F_q)$) if $g$ is quite big compared to $q$.

 Indeed they first recovered, for $n=1$, the classical Weil bound, that can  be seen as a \textit{first-order Weil bound}:
$$\sharp X(\mathbb F_q)\leq q+1+2g\sqrt{q}.$$

 For $n=2$, they found again the Ihara bound (to which they refered as the  \textit{second-order Weil bound}): if $g\geq g_2:=\frac{\sqrt{q}(\sqrt{q}-1)}{2}$ then
 $$\sharp X(\mathbb F_q)\leq q+1+\frac{\sqrt{(8q+1)g^2+4q(q-1)g}-g}{2}.$$

Finally, for $n=3$, they found a \textit{third-order Weil bound}: if $g\geq g_3:=\frac{\sqrt{q}(q-1)}{\sqrt{2}}$ then (Theorem 18 in \cite{H-P}):
$$\sharp X(\mathbb F_q)\leq q+1+\left(\sqrt{a(q)+\frac{b(q)}{g}+\frac{c(q)}{g^2}}-\frac{q-1}{q}+\frac{2\sqrt{q}(q-1)^2}{gq}\right)g\sqrt{q},$$
where
$$\left\{
\begin{array}{l}
a(q)=5+\frac{8}{\sqrt{q}}-\frac{1}{q^2}\\
b(q)=\frac{q-1}{q\sqrt{q}}\left(q^2-4q\sqrt{q}+2q+4\sqrt{q}-1\right)\\
c(q)=\frac{q-1}{4q}\left(q^3-5q^2-8q\sqrt{q}-5q-8\sqrt{q}+1\right).\\
\end{array}\right.$$

In a similar way, we would like to find better and better lower bounds for $x_2$ (possibly in function of $x_1$), in order to provide new upper bounds for $\sharp X(\mathbb F_{q^2})$. From each of these bounds we will deduce a new upper bound for the number of closed points  of degree 2 on $X$ and hence we will  be able to precise our equivalence  $(\ref{iff})$.

\section{Number of closed points of degree 2}



 Let $X$ be a smooth curve defined over $\mathbb F_q$ of genus $g$.  We recall that, if $B_2(X)$ denotes the number of closed points of degree 2 on $X$, one has
  $$B_2(X)=\frac{\sharp X(\mathbb F_{q^2})-\sharp X(\mathbb F_q) }{2}.$$

  \subsection{Upper bounds}
  We are going to establish upper bounds for the number $B_2(X)$  and then obtain upper bounds for
  the quantity $B_2({\mathcal X}_q(g))$  defined as the maximum number of closed points of degree 2 on an optimal smooth  curve of genus $g$ defined over ${\mathbb F}_q$.

  \subsubsection{The first order}
From the Weil bounds related to $(\ref{weil})$, we get
$\sharp X(\mathbb F_{q^2}) \leq q^2+1+2gq$ and $ \sharp X(\mathbb F_{q})\geq q+1 -2g\sqrt{q}$. Hence  an obvious upper bound for $B_2(X)$ is:
\begin{equation}\label{obvious}B_2(X)\leq \frac{q^2-q}{2}+g(q+\sqrt{q})=: M'(q,g). \end{equation}
We can consider $M'(q,g)$ as an upper bound  for $B_2(\mathcal X_q(g))$ at the first order since this bound is a direct consequence of the Weil bounds.

Using the quantity $M'(q,g)$,  we have recorded in the following table some first-order upper bounds for $B_2(\mathcal X_q(g))$  for specific couples $(q,g)$:  

\begin{center}
\begin{table}[H]
\begin{tabular}{|l|>{\centering\arraybackslash}p{1cm}|>{\centering\arraybackslash}p{1cm}|>{\centering\arraybackslash}p{1cm}|>{\centering\arraybackslash}p{1cm}|>{\centering\arraybackslash}p{1cm}|}
\hline
\backslashbox{q}{g}
&$2$&$3$&$4$&$5$&$6$\\\hline
$2$&$7$&$11$&$14$&$18$&$21$\\\hline
$3$&$12$&$17$&$21$&$26$&$31$\\\hline
$2^2$&$18$&$24$&$30$&$36$&$42$\\\hline
\end{tabular}
\bigskip
\caption{\footnotesize First-order upper bounds for  $B_2(\mathcal X_q(g))$ given by $M'(q,g)$.}
\end{table}
\end{center}

Unfortunately, the bound $(\ref{obvious})$ is rather bad,  so let us improve it.

We assume $g$ to be positive and   we consider $B_2(X)$ as a function of $x_1$ and $x_2$, defined in $(\ref{xi})$,  in the domain $\mathcal C_n \cap \mathcal W_n \cap \mathcal H_n^{q,g} $ to which $x_1$ and $x_2$ belong:
\begin{equation}\label{b2}B_2(X)=g\sqrt{q}(x_1-\sqrt{q}x_2)+\frac{q^2-q}{2}\end{equation}
since
 $\sharp X(\mathbb F_q)= q+1-2g\sqrt{q}x_1$ and
$\sharp X(\mathbb F_{q^2})=q^2+1-2gqx_2.$

We remark that any lower bound for $x_2$ implies an upper bound for $B_2(X)$, possibly in function of $x_1$.

We are going to investigate the set $\mathcal C_n \cap \mathcal W_n \cap \mathcal H_n^{q,g} $ introduced in the previous section   for $n=2$ (second order) and   $n=3$ (third order).

\subsubsection{The second order}
For $n=2$ the set $\mathcal C_2 \cap \mathcal W_2 \cap \mathcal H_2^{q,g}$ is given by the couples $(x_1,x_2)\in \mathbb R^2$ which satisfy the following system of inequalities:
\begin{equation}\label{system1}\left\{\begin{array}{l}
2x_1^2-1\leq x_2\leq 1\\
x_2\leq \frac{x_1}{\sqrt{q}}+\frac{q-1}{2g}.
\end{array} \right.\end{equation}

Geometrically it corresponds to the region of the plane $<x_1, x_2>$ delimited by the parabola $P: x_2=2x_1^2-1$ and  the lines $ L_2^{q,g}:x_2=\frac{x_1}{\sqrt{q}}+\frac{q-1}{2g}$ and  $x_2=1$. More precisely, depending on whether $g<g_2$, $g=g_2$ or $ g>g_2$, where $g_2=\frac{\sqrt{q}(\sqrt{q}-1)}{2}$, the region can assume  one  of the following three configurations:

\begin{table}[H]

\begin{center}
\begin{tabular}{lcr}
\includegraphics[height=4.8cm]{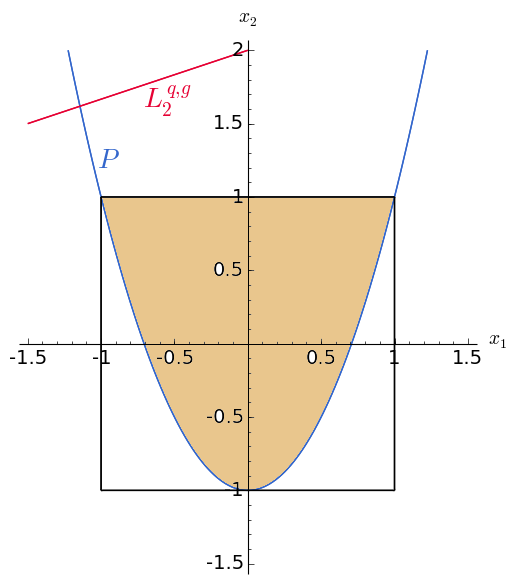} &\includegraphics[height=4.8
cm]{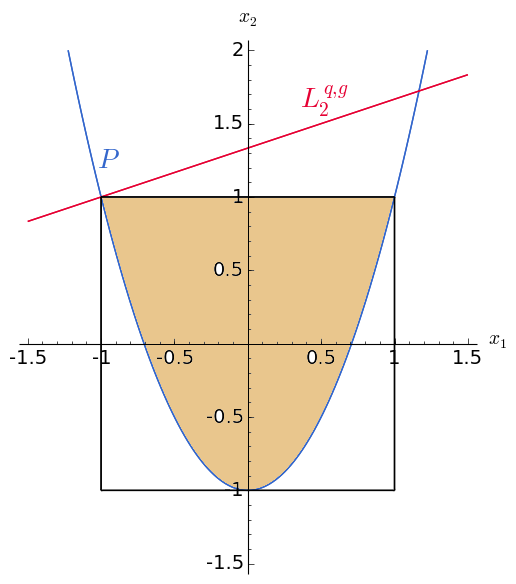}   &\includegraphics[height=4.8cm]{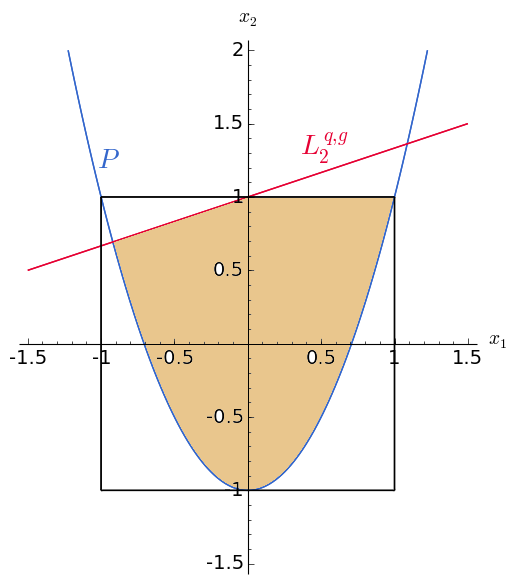}  \end{tabular}
\centering{
\caption{\footnotesize The region $\mathcal C_2\cap\mathcal W_2 \cap \mathcal H_2^{q,g}$, respectively for $g<g_2$, $g=g_2$ and $ g>g_2$.}}
\label{figure1}
\end{center}\end{table}
\normalsize
The first inequality in the system $(\ref{system1})$ \begin{equation}\label{boundx2}x_2\geq 2x_1^2-1,\end{equation} returns the upper bound:
\begin{equation}\label{expB2}B_2(X)\leq g\sqrt{q}(x_1-\sqrt{q}(2x_1^2-1))+\frac{q^2-q}{2}.\end{equation}
Expliciting $x_1$, we get the following bound for $B_2(X)$ in function of $q$,$g$ and $\sharp X(\mathbb F_q)$, which is a reformulation of Proposition 14 of \cite{H-P}:

\begin{proposition}\label{boundb2}
Let $X$ be a smooth curve of genus $g>0$ over $\mathbb F_q$. We have:
$$B_2(X)\leq \frac{q^2+1 +2gq-\frac{1}{g}\left(\sharp X(\mathbb F_q)- (q+1)\right)^2-\sharp X(\mathbb F_q)}{2}.$$
\end{proposition}

Now let us  suppose that $X$ is an optimal smooth curve of genus $g>0$, that is $X$ has $N_q(g)$ rational points. By Proposition $\ref{boundb2}$, if we set
$$M''(q,g):= \frac{q^2+1 +2gq-\frac{1}{g}\left(N_q(g)- (q+1)\right)^2-N_q(g)}{2},$$
then we have:
$$B_2(\mathcal X_q(g))\leq M''(q,g).$$
The quantity $M''(q,g)$ can hence be seen as a second-order upper bound for $B_2(\mathcal X_q(g))$.

We obtain the following proposition, as an easy consequence of $(\ref{iff})$:
\begin{proposition}\label{prop1} Let $g>0$.
If $\pi> g+ M''(q,g)$, then there do not exist $\delta$-optimal curves  defined over $\mathbb F_q$ of geometric genus $g$ and arithmetic genus $\pi$.\end{proposition}

In the following table we have used the quantity $M''(q,g)$ to get  upper bounds for $B_2(\mathcal X_q(g))$ for specific couples $(q,g$) (we have used  the datas about $N_q(g)$ which are available at \emph{http://www.manypoints.org/}).
\begin{center}
\begin{table}[!h]
\begin{tabular}{|l|>{\centering\arraybackslash}p{1cm}|>{\centering\arraybackslash}p{1cm}|>{\centering\arraybackslash}p{1cm}|>{\centering\arraybackslash}p{1cm}|>{\centering\arraybackslash}p{1cm}|}
\hline
\backslashbox{q}{g}
&$2$&$3$&$4$&$5$&$6$\\\hline
$2$&$1$&$2$&$3$&$4$&5\\\hline
$3$&$3$&$3$&$3$&$5$&$7$\\\hline
$2^2$&$5$&$0$&$4$&$5$&$3$\\\hline
\end{tabular}
\bigskip
\caption{\footnotesize Second-order upper bounds for  $B_2(\mathcal X_q(g))$ given by $M''(q,g)$.}
\label{T2}
\end{table}
\end{center}

\subsubsection{The third order} If now we increase the dimension to $n=3$, new constraints in $x_1,x_2 ,x_3$ are added to those of the system $(\ref{system1})$. Indeed, the set $\mathcal C_3 \cap \mathcal W_3 \cap \mathcal H_3^{q,g}$ is given by the triplets  $(x_1,x_2,x_3)\in \mathbb R^3$ which satisfy the following system of inequalities:

$$\left\{\begin{array}{l}
 2x_1^2-1\leq x_2\leq 1\\
-1+\frac{(x_1+x_2)^2}{1+x_1}\leq x_3 \leq 1-\frac{(x_1-x_2)^2}{1-x_1}\\
1+2x_1x_2x_3-x_3^2-x_1^2-x_2^2\geq 0\\
x_2\leq \frac{x_1}{\sqrt{q}}+\frac{q-1}{2g}\\
x_3\leq  \frac{x_1}{q}+\frac{q^2-1}{2g\sqrt q}.
\end{array} \right.$$

\bigskip
Let us consider the projection of $\mathcal C_3 \cap \mathcal W_3 \cap \mathcal H_3^{q,g}$ on the plane $<x_1, x_2>$, that is the set $\{(x_1, x_2)\in \mathbb R^2 : (x_1,x_2,x_3)\in \mathcal C_3 \cap \mathcal W_3 \cap \mathcal H_3^{q,g} \}$. It is easy to show that this set is given by the couples $(x_1,x_2)\in \mathbb R^2$ which satisfy:
\begin{equation}\label{system3}\left\{\begin{array}{l}
 2x_1^2-1\leq x_2\leq 1\\
-1+\frac{(x_1+x_2)^2}{1+x_1}\leq  \frac{x_1}{q}+\frac{q^2-1}{2g\sqrt q}\\
x_2\leq \frac{x_1}{\sqrt{q}}+\frac{q-1}{2g}.\\
\end{array} \right.\end{equation}

The   equation which corresponds to the second inequality in the system $(\ref{system3})$ is given by:
\begin{equation}\label{eq2}x_2^2+2x_1x_2-\left(\frac{1}{q}-1\right)x_1^2-\left(\frac{1}{q}+1+\frac{q^2-1}{2g\sqrt{q}}\right)x_1-1-\frac{q^2-1}{2g\sqrt{q}}= 0.\end{equation}
In the plane    $<x_1,x_2>$, the equation  $(\ref{eq2})$ represents an hyperbola $ H^{q,g}$  that passes through the point $(-1,1)$. For $g\geq g_3=\frac{\sqrt{q}(q-1)}{\sqrt{2}},$ the hyperbola $ H^{q,g}$ intersects the parabola in more than two  points. Hence we can have the following two configurations for the region of the plane which corresponds to the system $(\ref{system3})$:

\begin{table}[H]
\begin{center}
\begin{tabular}{lr}
\includegraphics[height=6cm]{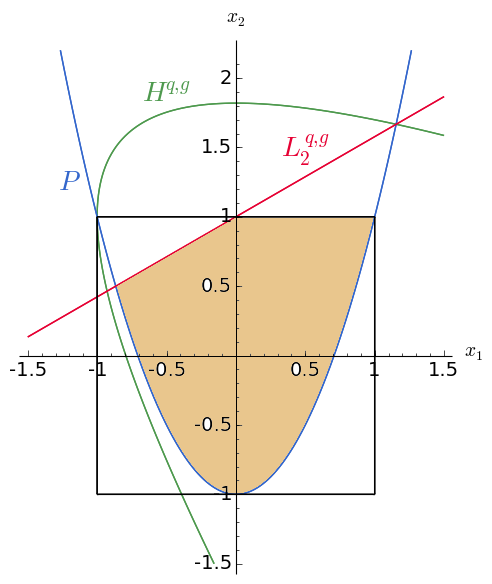} &\includegraphics[height=6
cm]{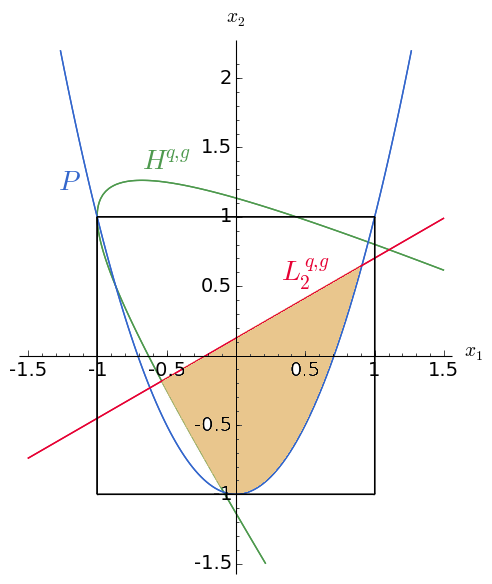}   \end{tabular}
\centering{
\caption{\footnotesize The  projection of $\mathcal C_3\cap\mathcal W_3 \cap \mathcal H_3^{q,g}$ on the plane $<x_1, x_2>$ respectively for $g<g_3$ and $ g>g_3$.}}
\end{center}
\label{figure2}
\end{table}

We remark that for $g\geq g_3$ we have a better  lower bound for $x_2$ in function of $x_1$ (compared to the bound  $(\ref{boundx2})$), which is given by the smallest solution  of the quadratic equation $(\ref{eq2})$ in $x_2$:

$$x_2\geq -x_1-\sqrt{\frac{1}{q}x_1^2+\left(\frac{1}{q}+1+\frac{q^2-1}{2g\sqrt{q}}\right)x_1+1+\frac{q^2-1}{2g\sqrt{q}}}$$
Thus, by $(\ref{b2})$, we get a new upper bound for $B_2(X)$, in function of $q$,$g$ and $x_1$:
  \begin{equation}\label{boundb2-2}\tiny B_2(X)\leq g\sqrt{q}\left((1+\sqrt{q})x_1+\sqrt{q}\sqrt{
\frac{1}{q}x_1^2+\left(\frac{1}{q}+1+\frac{q^2-1}{2g\sqrt{q}}\right)x_1+1+\frac{q^2-1}{2g\sqrt{q}}}\right)+\frac{q^2-q}{2}.
\end{equation}

Expliciting $x_1$ in $(\ref{boundb2-2})$, we get a new upper bound for $B_2(X)$ in function of $q$,$g$ and $\sharp X(\mathbb F_q)$:

\begin{proposition}\label{boundb2bis}
Let $X$ be a smooth curve of genus $g\geq\frac{\sqrt{q}(q-1)}{\sqrt{2}}$ over $\mathbb F_q$. We have:
\tiny
$$B_2(X)\leq \sqrt{1/4\left(\sharp X(\mathbb F_q)\right)^2+\alpha(q,g)\sharp X(\mathbb F_q)+\beta(q,g)}-\frac{(1+\sqrt{q})}{2}\sharp X(\mathbb F_q)+\frac{q^2+1+\sqrt{q}(q+1)}{2},$$
\normalsize where
$$\left\{
\begin{array}{l}
\alpha(q,g)=-\frac{1}{4}((2q\sqrt{q}+2\sqrt{q})g+q^3+q+2)\\
\beta(q,g)=\frac{1}{4}(4q^2g^2+2\sqrt{q}(q^3+q^2+q+1)g+q^4+q^3+q+1).\\
\end{array}\right.$$

\end{proposition}

As before, if we set
\tiny
$$M'''(q,g):=\sqrt{1/4\left(N_q(g)\right)^2+\alpha(q,g)N_q(g))+\beta(q,g)}-\frac{(1+\sqrt{q})}{2}N_q(g))+\frac{q^2+1+\sqrt{q}(q+1)}{2},$$
\normalsize
where $\alpha(q,g)$ and $\beta(q,g)$ are defined as in Proposition $\ref{boundb2bis}$, we have
$$B_2(\mathcal X_q(g))\leq M'''(q,g).$$
By $(\ref{iff})$, we get the following proposition:
\begin{proposition}\label{prop2}Let us assume that $g\geq \frac{\sqrt{q}(q-1)}{\sqrt{2}}$.
If $\pi> g+ M'''_q(g)$, then there do not exist $\delta$-optimal curves defined over $\mathbb F_q$  of geometric genus $g$ and arithmetic genus $\pi$.\end{proposition}



In the following table, using the quantity $M'''(q,g)$, we give upper bounds for $B_2(\mathcal X_q(g))$.  As $M'''(q,g)$ makes sense if and only if $g\geq \frac{\sqrt{q}(q-1)}{\sqrt{2}}$, some boxes of the table have been left empty.
\begin{center}
\begin{table}[!h]
\begin{tabular}{|l|>{\centering\arraybackslash}p{1cm}|>{\centering\arraybackslash}p{1cm}|>{\centering\arraybackslash}p{1cm}|>{\centering\arraybackslash}p{1cm}|>{\centering\arraybackslash}p{1cm}|}
\hline
\backslashbox{q}{g}
&$2$&$3$&$4$&$5$&$6$\\\hline
$2$&$0$&$0$&$1$&$1$&$1$\\\hline
$3$&&$2$&$1$&$2$&$3$\\\hline
$2^2$&&&&$4$&$1$\\\hline
\end{tabular}
\bigskip
\caption{\footnotesize Third-order upper bounds for  $B_2(\mathcal X_q(g))$  given by $M'''(q,g)$.}\label{T3}
\end{table}
\end{center}

Using Proposition $\ref{boundb2}$ and Proposition $\ref{boundb2bis}$, we can sum up Table $\ref{T2}$ and $\ref{T3}$ in the following one:

\begin{center}
\begin{table}[!h]
\begin{tabular}{|l|>{\centering\arraybackslash}p{1cm}|>{\centering\arraybackslash}p{1cm}|>{\centering\arraybackslash}p{1cm}|>{\centering\arraybackslash}p{1cm}|>{\centering\arraybackslash}p{1cm}|}
\hline
\backslashbox{q}{g}
&$2$&$3$&$4$&$5$&$6$\\\hline
$2$&$0$&$0$&$1$&$1$&$1$\\\hline
$3$&$3$&$2$&$1$&$2$&$3$\\\hline
$2^2$&$5$&$0$&$4$&$4$&$1$\\\hline
\end{tabular}
\bigskip
\caption{\footnotesize Upper bounds for  $B_2(\mathcal X_q(g))$.}
\end{table}
\end{center}

\subsection{Lower bound for $B_2(X)$}
In a similar way we can look for lower bounds for $B_2(X)$. From the Weil bounds related to $(\ref{weil})$, we have  $ \sharp X(\mathbb F_{q^2})\geq q^2+1-2gq$ and $ \sharp X(\mathbb F_{q}) \leq q+1+2g\sqrt{q}$ so that
\begin{equation}\label{lower}B_2(X)\geq \frac{q^2-q}{2}-g(q+\sqrt{q}).\end{equation}
It is easy to show that the quantity on the right-hand side of $(\ref{lower})$ is positive if and only if $g<g_2=\frac{\sqrt{q}(\sqrt{q}-1)}{2}$.

We can consider inequality $(\ref{lower})$ as a lower bound for $B_2(X)$ at the first order, as it is a direct consequence of the Weil bounds. Geometrically it is also clear that we will not obtain better lower bounds  at the second or  at the third order. Indeed, looking at the graphics in Table $2$ and Table $4$, we remark that, in some cases and for some values of $x_1$, a better upper bound for $x_2$ is  given by the line $L_2^{q,g}$. But we have seen in Remark $\ref{remark1}$ that if the pair $(x_1,x_2)$ is on the line $L_2^{q,g}$, then $\sharp X(\mathbb F_q)=\sharp X(\mathbb F_{q^2})$, which means $B_2(X)=0$.

\bigskip
For $g<g_2$, the inequality $(\ref{lower})$ implies the following lower bounds for  $B_2(\mathcal X_q(g))$:
\begin{center}
\begin{table}[!h]
\begin{tabular}{|l|>{\centering\arraybackslash}p{1cm}|>{\centering\arraybackslash}p{1cm}|>{\centering\arraybackslash}p{1cm}|>{\centering\arraybackslash}p{1cm}|}
\hline
\backslashbox{q}{g}
&$2$&$3$&$4$&$5$\\\hline
$7$&$2$&&&\\\hline
$2^3$&$7$&&&\\\hline
$3^2$&$12$&&&\\\hline
$11$&$27$&$13$&&\\\hline
$13$&$45$&$29$&$12$&\\\hline
$2^4$&$80$&$60$&$40$&$20$\\\hline
\end{tabular}
\bigskip
\caption{\footnotesize Lower bounds for  $B_2(\mathcal X_q(g))$.}
\end{table}
\end{center}

Hence we get  from the equivalence $(\ref{iff})$ and the inequality (\ref{lower}) the following proposition:
\begin{proposition}\label{prop3}
Let $g<\frac{\sqrt{q}(\sqrt{q}-1)}{2}$.
 If $g\leq \pi \leq g+ \frac{q^2-q}{2}-g(q+\sqrt{q})$, then there exists a $\delta$-optimal curve  defined over $\mathbb F_q$ of geometric genus $g$ and arithmetic genus $\pi$.\end{proposition}

\subsection{Some exact values for $N_q(g,\pi)$}

We can use the previous results to provide some exact values of $N_q(g,\pi)$ for specific triplets $(q,g,\pi)$.

\begin{proposition}\label{Nqgpi}
Let $q$ be a power of a prime number $p$. We have:
\begin{enumerate}
\item   $N_q(0,\pi)=q+1+\pi$ if and only if $0\leq \pi\leq \frac{q^2-q}{2}$ .
\item If $p$ does not divide $[2\sqrt{q}]$, or $q$ is a square, or $q=p$,  then $N_q(1,\pi)= q+[2\sqrt{q}]+\pi$  if and only if  $1 \leq \pi \leq 1+\frac{q^2+q-[2\sqrt{q}]([2\sqrt{q}]+1)}{2}$.

\noindent Otherwise,  $N_q(1,\pi)= q+[2\sqrt{q}]+\pi-1$   if and only if
$1 \leq \pi \leq 1+\frac{q^2+q+[2\sqrt{q}](1-[2\sqrt{q}])}{2}$.
\item  If  $g<\frac{\sqrt{q}(\sqrt{q}-1)}{2}$ and $g\leq \pi \leq \frac{q^2-q}{2}-g(q+\sqrt{q}-1)$ then $ N_q(g,\pi)= N_q(g)+\pi-g.$
\item $N_2(2,3)=6$.
\item $N_2(3,4)=7$.
\item $N_{2^2}(4,5)=14.$
\end{enumerate}
\end{proposition}

\begin{proof}

Items {\it (1)}  and {\it(2)} are Corollary 5.4 and Corollary 5.5 in \cite{A-I}.
Item {\it (3)} is given by Proposition $\ref{prop3}$.

We have that $N_2(2,3)\geq N_2(2)=6$ and $B_2(\mathcal X_2(2))=0$, by Table 6. Hence {\it(4)} follows from Proposition $\ref{prop2}$ which says that $N_2(2,3)<N_2(2)+1$. Items {\it (5)} and {\it(6)} can be proven in the same way.

\end{proof}

\begin{remark}
Using the construction given in Section 3 of \cite{A-I}, we can easily show that $N_q(g,\pi+1)\geq N_q(g,\pi)$. This fact implies, for instance, that we have also $N_q\left(0,\frac{q^2-q}{2}+1\right)=q+1+ \frac{q^2-q}{2}$.
 \end{remark}

\section{Genera spectrum of maximal curves}\label{spectrum}


Let $X$ be a curve defined over $\mathbb F_q$ of geometric genus $g$ and arithmetic genus $\pi$. We recall that $X$ is a maximal curve if it attains  bound (\ref{APbound}), i.e
 $$\sharp{X}({\mathbb F}_q)= q+1+g[2\sqrt{q}]+ \pi-g.$$
 This definition recovers the classical definition of a smooth maximal curve.

An easy consequence of Proposition 5.2 in \cite{A-I} is that if $X$ is a maximal curve, then its normalization $\tilde{X}$ is a smooth maximal curve. Moreover,
the zeta function of a maximal curve $X$ is given by (see Prop. 5.8 in \cite{A-I}):
$$Z_{X}(T)=Z_{\tilde{X}}(T)(1+T)^{\pi-g}=\frac{(qT^2+[2\sqrt{q}]T+1)^g(1+T)^{\pi-g}}{(1-T)(1-qT)}.$$

We have seen in the previous section that, for $\pi$ quite big compared to $g$,  maximal curves of geometric genus $g$ and arithmetic genus $\pi$ don't exist.

 Hence, a related question  concerns the genera spectrum of maximal curves defined over $\mathbb F_q$, i.e. the set of couples $(g,\pi)$, with $g,\pi \in \mathbb N $ and $g\leq \pi$, for which  there exists a maximal curve over $\mathbb F_q $  of geometric genus  $g$ and arithmetic genus $\pi$:
\begin{align*}\Gamma_q:=&\{(g,\pi) \in \mathbb N\times \mathbb N\textrm{ : there exists a maximal curve defined over }\mathbb F_q    \\ &\textrm{ of geometric
genus   }g \textrm{ and arithmetic genus }\pi\}.\end{align*}

The analogous question in the smooth case has been  extensively  studied in the case where $q$ is a square. For $q$ square,  Ihara proved that if $X$ is a maximal smooth curve defined over $\mathbb F_q$ of genus $g$, then  $g\leq\frac{\sqrt{q}(\sqrt{q}-1)}{2}$  (see \cite{Ihara}) and R\" uck and Stichtenoth showed that $g$ attains this upper bound if and only if $X$ is $\mathbb F_{q}$-isomorphic to the Hermitian curve (see \cite{R-S}). Moreover, Fuhrmann and Garcia proved  that the genus $g$ of maximal smooth curves defined over $\mathbb F_q$ satisfies (see \cite{F-T})

\begin{equation}\label{gap1}\textrm{either} \qquad g\leq \left\lfloor\frac{(\sqrt{q}-1)^2}{4}\right\rfloor, \qquad \textrm{or} \qquad g=\frac{\sqrt{q}(\sqrt{q}-1)}{2}.\end{equation}

This fact corresponds to the so-called \emph{first gap} in the spectrum genera of $\mathbb F_q$-maximal smooth curves. For $q$ odd, Fuhrmann, Garcia and Torres showed that  $g=\frac{(\sqrt{q}-1)^2}{4}$ occurs if and only if $X$ is $\mathbb F_q$-isomorphic to the non-singular model of the plane curve of equation $y^{\sqrt{q}}+y=x^{\frac{\sqrt{q}+1}{2}}$ (see \cite{FGT}). For $q$ even, Abd\'on and Torres established a similar result  in \cite{A-T} under an extra-condition that $X$ has a particular Weierstrass point. In this case, $g=\frac{\sqrt{q}(\sqrt{q}-2)}{4}$ if and only if $X$ is $\mathbb F_q$-isomorphic to the non-singular model of the plane curve of equation $y^{\sqrt{q}/2}+\cdots +y^2+y=x^{(\sqrt{q}+1)}$.

Korchm\'aros and Torres improved $(\ref{gap1})$ in \cite{K-T}:
\begin{equation}\label{gap2} \textrm{either} \,g\leq  \left\lfloor\frac{q-\sqrt{q}+4}{6} \right\rfloor,  \, \textrm{or} \,\,g\leq  \left\lfloor\frac{(\sqrt{q}-1)^2}{4} \right\rfloor,\, \textrm{or}\,\,  g=\frac{\sqrt{q}(\sqrt{q}-1)}{2}.\end{equation}
Hence  the \emph{second gap} in the spectrum  genera of $\mathbb F_q$-maximal smooth curves is also known. In the same paper, non-singular $\mathbb F_q$-models of genus $ \left\lfloor\frac{q-\sqrt{q}+4}{6} \right\rfloor$ are provided.

\bigskip

Let us now consider  maximal curves with possibly singularities.
We assume $q$ square and we want to study the discrete set $\Gamma_q$.

Let $X$ be a maximal curve defined over $\mathbb F_q$ of geometric genus $g$ and arithmetic genus $\pi$. As remarked above, the normalization $\tilde{X}$ of $X$ is a maximal smooth curve, hence we have that $g$ satisfies $(\ref{gap2})$. Moreover,  $g$ and $\pi$ verify the following inequality:

\begin{proposition}\label{max} Let $q$ be a square. There exists a maximal curve defined over $\mathbb F_q$  of geometric genus $g$ and arithmetic genus $\pi $
 if and only if $ N_q(g)=q+1+2g\sqrt{q}$ and \begin{equation}\label{bound}g\leq\pi\leq g+ \frac{q^2+(2g-1)q-2g\sqrt{q}(2\sqrt{q}+1)}{2}.\end{equation}
\end{proposition}

\begin{proof}
The proposition follows directly from  the equivalence $(\ref{iff})$, from the fact that a maximal curve has a maximal normalization and  that the number of closed points of degree $2$ on a smooth maximal curve of genus $g$ over $\mathbb F_q$ is given by
(see Prop. 5.8 of \cite{A-I}): $\frac{q^2+(2g-1)q-2g\sqrt{q}(2\sqrt{q}+1)}{2}$.
\end{proof}

\begin{remark}\label{majorant_de_pi}
The quantity on the right-hand side  of (\ref{bound}), that can be written as  $$ (-q-\sqrt q+1)g+\frac{q^2-q}{2},$$ is a linear decreasing function in the variable $g$. Hence it attains its maximum value for $g=0$ (this also means that the number of closed points of degree $2$ on a maximal smooth curve decreases with the increasing of the genus). So we get also a bound for the arithmetic genus $\pi$ in terms of the cardinality of  the finite field:
 \begin{equation}\label{boundpi}\pi\leq \frac{q(q-1)}{2}.\end{equation}

\end{remark}

Geometrically, we have shown that the set $\Gamma_q$ is  contained in the triangle $(OAB)$ (see Figure \ref{fig:triangle}) of  the plane $<g,\pi>$ delimited by the lines $g=0$, $\pi=(-q-\sqrt q+1)g+\frac{q^2-q}{2}$ and $g= \pi$.

We observe that maximal curves over $\mathbb F_q$ with geometric genus $g=\frac{\sqrt{q}(\sqrt{q}-1)}{2}$ are necessarily smooth and thus isomorphic to the Hermitian curve.

Also the bound (\ref{boundpi}) is  sharp. Indeed the singular plane rational curve  provided in \cite{FHK} is an example of a maximal curve defined over $\mathbb F_q$ with arithmetic genus $\pi= \frac{q(q-1)}{2}$.

Hence, using Proposition \ref{Nqgpi}, the inequalities (\ref{gap2}), Proposition \ref{max} and  Remark \ref{majorant_de_pi}, we can state the following theorem:
\begin{theorem}\label{gpi}
Let $q$ be a square and $X$ be a maximal curve defined over $\mathbb F_q$  with geometric genus $g$ and arithmetic genus $\pi$.

If we set $g':=\frac{\sqrt{q}(\sqrt{q}-1)}{2}$, $g'':= \left\lfloor \frac{(\sqrt{q}-1)^2}{4}\right\rfloor$ and $g''':=\left\lfloor\frac{q-\sqrt{q}+4}{6} \right\rfloor$,
then we have:
\begin{enumerate}
\item $0\leq g\leq g'$ and $g\leq\pi\leq \frac{q(q-1)}{2}$ and also $\pi\leq g+ \frac{q^2+(2g-1)q-2g\sqrt{q}(2\sqrt{q}+1)}{2}$. In other words  $\Gamma_q$ is contained in the set of integer points inside the triangle $(OAB)$.

\vspace{0.3 cm}
\item The point $B=(g',g')$ belongs to $\Gamma_q$ and the set of points
$$\left\{\left( 0,\pi\right), \textrm{ with } 0\leq \pi \leq  \frac{q^2-q}{2}\right\}$$
is contained in $\Gamma_q$.

\item If $g\neq g'$ then $g\leq g''$ and  the set of points
$$\left\{\left( g'',\pi\right), \textrm{ with } g''\leq \pi \leq  (-q-\sqrt q+1)g''+\frac{q^2-q}{2}\right\}$$
is contained in $\Gamma_q$.

\vspace{0.3 cm}
\item If $g\neq g'$ and $g\neq g''$, then $g \leq g'''$ and  the set of points
$$\left\{\left( g''',\pi\right), \textrm{ with } g'''\leq \pi \leq  (-q-\sqrt q+1)g'''+\frac{q^2-q}{2}\right\}$$
is contained in $\Gamma_q$.

\end{enumerate}
\end{theorem}

We can illustrate Theorem \ref{gpi} with the following  figure (in which the aspect ratio has been chosen equal to 0.025):
\begin{figure}[H]
\includegraphics[scale=0.25]{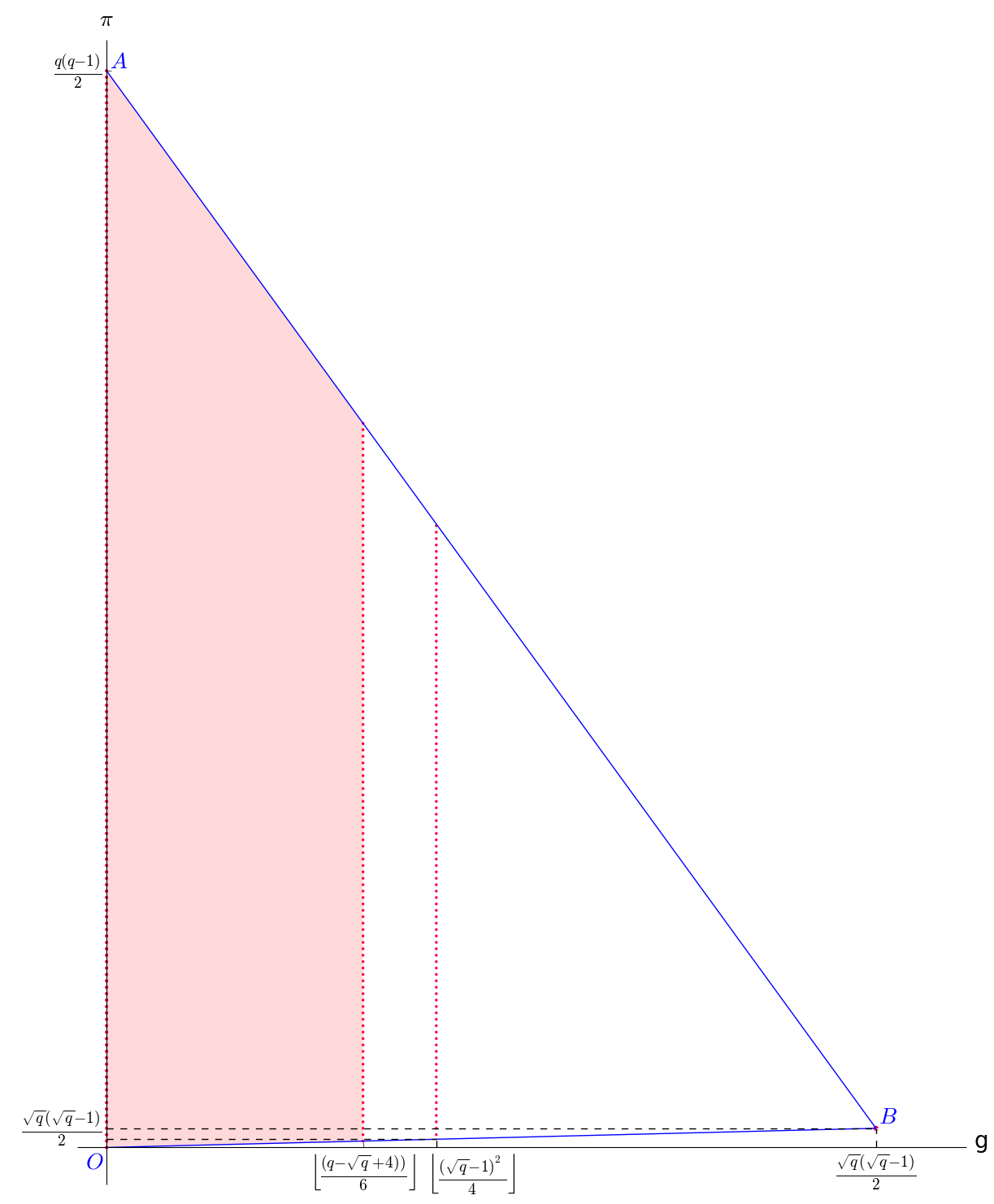}
 \caption{{\footnotesize{The set $\Gamma_q$ is contained in the set of integers points inside the triangle $(OAB)$. The dots correspond to the couples $(g,\pi)$ that we have proved to be in $\Gamma_q$. The rest of the set $\Gamma_q$ has to be contained in the colored trapezoid. }}
}
\label{fig:triangle}
\end{figure}

\bigskip
\bigskip

We conclude the paper by considerations on coverings of singular curves.

If $f:Y\rightarrow X$ is a surjective morphism of smooth curves defined over $\mathbb F_q$ and if $Y$ is maximal then $X$ is also maximal. This result is due to Serre (see \cite{lachaud}).
We prove here that the result still holds without the smoothness hypothesis of the curves but with the flatness hypothesis of the morphism.
Remark that the divisibility of the numerators of the zeta functions in a flat covering proved in \cite{A-P_Archiv} for possibly singular curves and in \cite{A-P1} for possibly singular varieties does not lead the result.

\begin{theorem}
Let $f:Y\rightarrow X$ be a finite flat morphism between two curves defined over  $\mathbb F_q$. If $Y$ is maximal then $X$ is maximal.
\end{theorem}

\begin{proof}
Let us denote by $g_X$  and $\pi_X$ (respectively $g_Y$ and $\pi_Y$)  the geometric genus and the arithmetic genus of $X$ (respectively of $Y$).
As $Y$ is maximal, we have
$$\sharp Y(\mathbb F_q)=q+1+g_Y[2\sqrt{q}]+\pi_Y-g_Y.$$
From Remark 4.1 of \cite{A-P2} we know that
$$|\sharp Y(\mathbb F_q)-\sharp X(\mathbb F_q)|\leq (\pi_Y-g_Y)-(\pi_X-g_X)+(g_Y-g_X)[2\sqrt{q}].$$
So we obtain:
\begin{align*}\sharp X(\mathbb F_q)&\geq \sharp Y(\mathbb F_q) -(\pi_Y-g_Y)+(\pi_X-g_X)-(g_Y-g_X)[2\sqrt{q}]\\&= q+1+g_X[2\sqrt{q}]+ \pi_X-g_X.\end{align*}
Hence $X$ is also maximal.
\end{proof}

\bibliography{biblio} 
\bibliographystyle{plain}

\end{document}